\documentclass[11pt, amsfonts]{amsart}

\usepackage{amsmath,amssymb,amsthm,enumerate,comment}
\usepackage[all]{xy}
\usepackage{color}
\usepackage{graphicx}
\usepackage[T1]{fontenc}
\usepackage[dvipdfm,colorlinks=true]{hyperref}
\usepackage[colorlinks=true]{hyperref}

\SelectTips{eu}{12}

\theoremstyle{definition}
\newtheorem{theorem}{Theorem}[section]
\newtheorem{definition}[theorem]{Definition}
\newtheorem{lemma}[theorem]{Lemma}
\newtheorem{proposition}[theorem]{Proposition}
\newtheorem{corollary}[theorem]{Corollary}

\newtheorem{remark}[theorem]{\rm Remark}

\makeatletter
\@addtoreset{equation}{section} 
\makeatother

\DeclareMathOperator{\id}{id}

\DeclareMathOperator{\PSL}{PSL}

\newcommand{\RR}{\mathbb{R}}
\newcommand{\QQ}{\mathbb{Q}}
\newcommand{\ZZ}{\mathbb{Z}}

\newcommand{\Supp}{\mathrm{Supp}}
\newcommand{\PP}{\mathrm{PP}_+(\RR)}
\newcommand{\I}{\mathcal{I}}
\newcommand{\hf}{\widehat{f}}
\newcommand{\PL}{\mathrm{PL}_+([0,1])}

\makeatletter
\@namedef{subjclassname@2020}{%
  \textup{2020} Mathematics Subject Classification}
\makeatother

\title[Invariable generation of groups of piecewise projective homeomorphisms]{Invariable generation of certain groups of piecewise projective homeomorphisms of the real line}
\author{Shuhei Maruyama}
\address{School of Mathematics and Physics, College of Science and Engineering, Kanazawa University, Kakuma-machi, Kanazawa, Ishikawa, 920-1192, Japan}
\email{smaruyama@se.kanazawa-u.ac.jp}

\keywords{invariable generation; groups of piecewise projective homeomorphisms}

\begin{document}

\begin{abstract}
  We show that the following groups are invariably generated; the group of piecewise projective homeomorphisms of the real line, the group of piecewise $\mathrm{PSL}(2,\mathbb{Z})$ homeomorphisms of the real line, Monod's group $H(\mathbb{Z})$, the group of piecewise $\mathrm{PSL}(2,\mathbb{Q})$ homeomorphisms of the real line with rational breakpoints.
  We also show that the Higman--Thompson group $F_n$ for every $n \in \mathbb{Z}_{\geq 3}$ and the golden ratio Thompson group $F_{\tau}$ are invariably generated.
\end{abstract}

\maketitle

\section{Introduction}

A group $G$ is said to be \emph{invariably generated} if for every map $\chi \colon G \to G$, the set
\[
  \{ g^{\chi(g)} \colon g \in G \}
\]
generates $G$.
Here we use the symbol $g^h$ to denote $hgh^{-1}$.
The concept of invariable generation appeared in \cite{MR417300} (in a different language), and the term \emph{Invariable generation} was first used in \cite{MR1180190}.
A subgroup $H$ of $G$ is said to be \emph{classful} if $H$ intersects every conjugacy class of $G$.
It is easy to show that a group is invariably generated if and only if there are no proper subgroups which are classful.

It is easily verified that every abelian group is invariably generated.
Every finite group is also invariably generated (\cite{MR2852243}).
Wiegold showed in \cite{MR417300} that invariable generation is closed under extension, that is, for a group $G$, the invariable generation of a normal subgroup $N$ of $G$ and that of the quotient group $G/N$ imply that of $G$.
It follows that every virtually solvable group is invariably generated (see \cite{MR3272383} for example).
Contrastingly, the nonabelian free groups are not invariably generated (\cite{MR417300}).

In \cite{MR3272383}, Kantor, Lubotzky and Shalev studied invariable generation of linear groups.
After that, invariable generation of several infinite groups has been investigated.
For example, the following groups were shown to be invariably generated; the Grigorchuk group (\cite{MR3272383}), Richard Thompson's group $F$ (\cite{MR3621672}), the group $\PL$ of PL homeomorphisms of the interval $[0,1]$ and certain subgroups of it (\cite{MR4172581}), the Houghton groups (\cite{MR4379277}).
The following groups were shown to be not invariably generated; any nonelementary convergence group (\cite{MR3431613}), any acylindrically hyperbolic group (\cite{MR3715440}), Richard Thompson's groups $T$ and $V$ (\cite{MR3621672}) and the Higman--Thompson groups $T_n$ and $V_n$ for any $n \geq 2$ (\cite{2207.04235}).

Richard Thompson's group $F$ is known to have no nonabelian free subgroups, and its amenability is a major open question.
The von Neumann--Day problem asks whether every nonamenable group contains a nonabelian free subgroup; its counterexample was constructed by Ol'\v{s}anski\u{\i} (\cite{MR586204}).

Monod provided in \cite{MR3047655} a family of dynamically tractable counterexamples to the von Neumann--Day problem.
One of them is the following.
A homeomorphism $f$ of $\RR P^1 = \RR \cup \{ \infty \}$ is called a \emph{piecewise projective homeomorphism} if there exists a family of finitely many closed intervals of $\RR P^1$ which cover $\RR P^1$ and the restriction of $f$ to each interval coincides with a M\"{o}bius transformation of an element of $\PSL(2,\RR)$ on $\RR P^1$.
For a piecewise projective homeomorphism $f$ satisfying $f(\infty) = \infty$, the restriction $f|_{\RR}$ is called a \emph{piecewise projective homeomorphism of the real line}.
Let $\PP$ be the group of piecewise projective homeomorphisms of the real line.
Monod showed in \cite{MR3047655} that the group $\PP$ is not amenable and does not contain any nonabelian free subgroup, that is, it is a counterexample of the von Neumann--Day problem.

Let $K$ be a subgroup of $\PSL(2,\RR)$.
If the restrictions of $f \in \PP$ to each interval is a M\"{o}bius transformation of an element of $K$, then we call $f$ a \emph{piecewise $K$ homeomorphism of the real line}.

In \cite{MR4172581}, Matsuda and Matsumoto established a criterion for the invariable generation of a given subgroup of $\PL$ (see Theorem \ref{thm:MMthm_cond_inv_gen}).
Introducing and using an analogous criterion for the invariable generation of a subgroup of $\PP$, we show the following.

\begin{theorem}\label{mainthm}
  The following groups are invariably generated;
  \begin{enumerate}[$(1)$]
    \item the group $\PP$ of piecewise projective homeomorphisms of the real line,
    \item the group of piecewise $\PSL(2,\ZZ)$ homeomorphisms of the real line,
    \item Monod's group $H(\ZZ)$,
    \item the group of piecewise $\PSL(2,\QQ)$ homeomorphisms of the real line with rational breakpoints.
  \end{enumerate}
\end{theorem}

Here Monod's group $H(\ZZ)$ is the group of piecewise $\PSL(2,\ZZ)$ homeomorphisms of the real line whose breakpoints are contained in
\[
  P_{\ZZ} = \{ x \in \RR \colon \exists \text{ hyperbolic element } A \in \PSL(2,\ZZ) \text{ such that } Ax = x \}.
\]
Monod's group $H(\ZZ)$ does not contain any nonabelian free subgroup, and its amenability is unknown (\cite[Problem 12]{MR3047655}).

In Appendix \ref{sec:appendix}, slightly modifying the criterion for the invariable generation of a subgroup of $\PL$ given in \cite{MR4172581}, we show the following.
\begin{theorem}\label{thm:appendix_main}
  The following groups are invariably generated;
  \begin{enumerate}[$(1)$]
    \item the Higman--Thompson group $F_n$ for every $n \in \ZZ_{\geq 3}$,
    \item the golden ratio Thompson group $F_{\tau}$.
  \end{enumerate}
\end{theorem}

Theorem \ref{thm:appendix_main} (1) stands in contrast to the fact that the Higman--Thompson groups $T_n$ and $V_n$ are not invariably generated for any $n \in \ZZ_{\geq 2}$ (\cite{MR3621672} for $n = 2$, and \cite{2207.04235} for any $n \in \ZZ_{\geq 2}$).
Note also that in (1), the Higman--Thompson group $F_2$ is just Richard Thompson's group $F$, and its invariable generation was shown in \cite{MR3621672} and \cite{MR4172581}.
To show Theorems \ref{mainthm} and \ref{thm:appendix_main}, we follow the line of the proofs of invariable generation given in \cite{MR4172581}.

\section{Condition for invariable generation}\label{sec:cond_for_inv_gen}
Since an element $f$ of $\PP$ preserves $\pm \infty$, the function $f$ has linear germs at $\pm \infty$.
For $f \in \PP$ which is not the identity map, the maximal interval $(-\infty, s]$ (resp. $[t, \infty)$) on which $f$ is linear is called the \emph{end linear zone of $f$ at $-\infty$} (resp. $\infty$).
For the identity map, we define both end linear zones of it to be $\RR$.
Let $\PP_1$ be the subgroup of $\PP$ whose elements have slope $1$ on both end linear zones.
Note that the group $\PP_1$ is a normal subgroup of $\PP$ that contains the commutator subgroup $\PP'$ of $\PP$.

Let $G$ be a subgroup of $\PP_1$ and $X$ a subset of the real line $\RR$.
A closed interval $I$ in $\RR$ is called an \emph{$X$-interval} if the endpoints of $I$ are contained in $X$.
A closed interval $I$ in $\RR$ is called an \emph{$X^*$-interval} if it is an $X$-interval or there exists an element $x \in X$ such that $I$ is of the form $[x, \infty)$ or $(-\infty, x]$.
For a subgroup $H$ of $G$ and an $X^*$-interval $I$, we set
\begin{align*}
  H(I) &:= \{ h|_I \colon h \in H, \Supp(h) \subset I \}, \\
  H|_I &:= \{ h|_I \colon h \in H, h(I) = I \}.
\end{align*}
Here $\Supp(f)$ denotes the support of $f$, that is, the closure of $\{ t \in \RR \mid f(t) \neq t \}$.

The following Conditions are analogous to ones given by Matsuda--Matsumoto (\cite{MR4172581}, see also Definitions \ref{def:MMCondA} and \ref{def:MMCondB}).
\begin{definition}
  Let $G$ be a subgroup of $\PP_1$ and $X$ a subset of $\RR$.
  We say that \emph{$G$ satisfies Condition A' with respect to $X$} if the following hold:
  \begin{enumerate}
    \item[A'$(1)$] $X$ is a $G$-invariant subset of $\RR$ which contains all the breakpoints of every element of $G$.
    \item[A'$(2)$] The orbit $G\cdot x \subset \RR$ is not bounded from above nor below for any $x \in X$.
    \item[A'$(3)$] For every $X^*$-interval $I$, the equality $G|_I = G(I)$ holds.
  \end{enumerate}
\end{definition}

\begin{definition}
  Let $G$ be a subgroup $\PP_1$ and $X$ a subset of $\RR$.
  We say that \emph{$G$ satisfies Condition B' with respect to $X$} if the following hold:
  \begin{enumerate}
    \item[B'$(1)$] For every classful subgroup $H$ of $G$ and every $x \in X$, the action of $H$ on the orbit $G \cdot x$ is transitive.
    \item[B'$(2)$] For every classful subgroup $H$ of $G$, every $X^*$-interval $I$ of the form $[v, \infty)$, and every $x \in I\cap X$, the action of $H(I)$ on $(G \cdot x) \cap (v,\infty)$ is transitive.
    \item[B'$(3)$] For every classful subgroup $H$ of $G$, there exists an $X$-interval $I_0$ such that $H|_{I_0} = G(I_0)$.
  \end{enumerate}
\end{definition}

Conditions A' and B' provides a criterion for the invariable generation of $G$ as follows.
\begin{theorem}\label{thm:conditions}
  Let $G$ be a subgroup of $\PP_1$.
  If there exists $X \subset \RR$ such that $G$ satisfies Conditions A' and B' with respect to $X$, then $G$ is invariably generated.
\end{theorem}

The proof is done in a similar way as in \cite{MR4172581}.
For the convenience of the reader, we provide the proof.

Let $G$ be a subgroup of $\PP_1$ and $X$ a subset of $\RR$.
For $f \in G$, we set $s(f) := \sup\{ t \colon f|_{(-\infty, t]} = \id \}$ and $i(f) = \inf \{ t \colon f|_{[t, \infty)} = \id \}$.
\begin{lemma}\label{lem:classful_inf_int}
  Assume that $G$ satisfies Conditions A'(1), A'(3), and B'(1) with respect to $X$.
  Let $H$ be a classful subgroup of $G$ and $I = [v, \infty)$ an $X^*$-interval.
  Then, $H(I)$ is a classful subgroup of $G(I)$.
\end{lemma}

\begin{proof}
  Let $f$ be an element of $G$ with $\Supp(f) \subset I$.
  It suffices to show that there exists an element $g$ of $G$ with $\Supp(g) \subset I$ such that $f^g \in H$.
  As the case where $f$ is the identity map is clear, we assume that $f$ is not the identity map.
  Then we have $s(f) \geq v$.
  Since $H$ is classful in $G$, there exists $g_1 \in G$ such that $f^{g_1} \in H$.
  Then we have $s(f^{g_1}) = g_1(s(f))$.
  Note that by Condition A'(1), the element $s(f)$ is contained in $X$.
  Since $s(f)$ and $g_1(s(f))$ are contained in the orbit $G\cdot s(f)$, there exists $g_2 \in H$ such that $g_2(g_1(s(f))) = s(f)$ by Condition B'(1).
  Here $f^{g_1}$ and $g_2$ are contained in $H$, and hence so is $f^{g_2g_1} = (f^{g_1})^{g_2}$.
  Moreover, since $g_2g_1$ fixes $s(f)$, we have $g_2g_1|_{[s(f), \infty)} \in G|_{[s(f), \infty)}$.
  By Condition A'(3), there exists $g \in G$ such that
  \[
    g =
    \begin{cases}
      \id & \text{ on } \ (-\infty, s(f)] \\
      g_2g_1 & \text{ on } \  [s(f), \infty).
    \end{cases}
  \]
  By definition, we have $f^g = f^{g_2g_1} \in H$ and the support of $g$ is contained in $[s(f), \infty) \subset I$.
  Hence this $g \in G$ is a desired one.
\end{proof}

\begin{lemma}\label{lem:classful_X_int}
  Assume that $G$ satisfies Conditions A'(1), A'(3), B'(1), and B'(2) with respect to $X$.
  Let $H$ be a classful subgroup of $G$ and $I$ an $X$-interval.
  Then, the group $H(I)$ is a classful subgroup of $G(I)$.
\end{lemma}

\begin{proof}
  The proof is the same as that of Lemma \ref{lem:classful_inf_int}.

  We set $I = [v,w]$ and $J = [v, \infty)$.
  Let $f$ be an element of $G$ with $\Supp(f) \subset I$.
  It suffices to show that there exists an element $g$ of $G$ with $\Supp(g) \subset I$ such that $f^g$ is contained in $H$.

  Since the case where $f = \id$ is clear, we assume that $f$ is not the identity map.
  Note that, since $f$ is not the identity map, we have $v < i(f) \leq w$.
  By Condition A'(1), we have $i(f) \in J \cap X$.
  Since the group $H(J)$ is classful in $G(J)$ by Lemma \ref{lem:classful_inf_int}, we take an element $g_1 \in G$ with $\Supp(g_1) \subset J$ such that $f^{g_1}$ is contained in $H$.
  Note that $g_1(i(f)) = i(f^{g_1}) \in J\cap X$.
  Moreover, $g_1(i(f))$ is greater than $v$.
  Hence, together with the fact that $i(f)$ and $g_1(i(f))$ lie in the same orbit $G \cdot i(f)$, Condition B'(2) asserts that there exists $g_2 \in H$ with $\Supp(g_2) \subset J$ such that $g_2(g_1(i(f))) = i(f)$.
  Since $f^{g_1}$ and $g_2$ are both in $H$, so is $f^{g_2g_1}$.
  As $g_2g_1$ fixes $v$ and $i(f)$, we have $g_2g_1|_{[v, i(f)]} \in G(I)$.
  By Condition A'(3), there exists $g \in G$ such that
  \[
    g =
    \begin{cases}
      \id & \text{ on } \ (-\infty, v] \cup [i(f), \infty) \\
      g_2g_1 & \text{ on } \  [v, i(f)].
    \end{cases}
  \]
  By definition, we have $f^g = f^{g_2g_1} \in H$ and the support of $g$ is contained in $[v, i(f)] \subset I$.
  Hence this $g \in G$ is a desired one.
\end{proof}

\begin{lemma}\label{lem:equal_X_interval}
  Assume that $G$ satisfies Conditions A'(1), A'(3), B'(1), and B'(2) with respect to $X$.
  Let $H$ be a classful subgroup of $G$ and $I_0$ an $X$-interval satisfying $H|_{I_0} = G(I_0)$. Then, the equality $H(I_0) = G(I_0)$ holds.
\end{lemma}

\begin{proof}
  We shall show $G(I_0) \subset H(I_0)$.
  Let $f$ be an element of $G$ with $\Supp(f) \subset I_0$.
  It suffices to show that $f$ is contained in $H$.
  Since $H(I_0)$ is classful in $G(I_0)$ by Lemma \ref{lem:classful_X_int}, there exists $g \in G$ with $\Supp(g) \subset I_0$ such that $f^g \in H$.
  As $g|_{I_0} \in G(I_0) = H|_{I_0}$, there exists $h \in H$ such that $h|_{I_0} = g|_{I_0}$.
  Note that the equality $f^h = f^g$ holds since $\Supp(f) \subset I_0$.
  Hence we obtain $f = (f^g)^{h^{-1}}$, which is contained in $H$.
\end{proof}

\begin{lemma}\label{lem:double_trans}
  Assume that $G$ satisfies Conditions A'(1), B'(1), and B'(2) with respect to $X$.
  Let $H$ be a classful subgroup of $G$.
  Let $x_0$ and $y_0$ be elements of $X$ with $x_0 < y_0$ and $x_1$ and $y_1$ elements of $G\cdot x_0$ and $G\cdot y_0$, respectively, with $x_1 < y_1$.
  Then, there exists an element $h$ of $H$ such that $h(x_0) = x_1$ and $h(y_0) = y_1$.
\end{lemma}

\begin{proof}
  By Condition B'(1), we take $h_1 \in H$ such that $h_1(x_0) = x_1$.
  By Condition A'(1), the element $x_1$ is contained in $X$.
  Condition B'(2) applied to the $X^*$-interval $[x_1, \infty)$ states that the action of $H([x_1, \infty))$ on $(G\cdot h_1(y_0)) \cap (x_1, \infty)$ is transitive.
  Hence, together with $y_1 \in G \cdot y_0 = G \cdot h_1(y_0)$ and $y_1 \in (x_1, \infty)$, there exists $h_2 \in H$ such that $\Supp(h_2) \subset [x_1, \infty)$ and $h_2(h_1(y_0)) = y_1$.
  Then the element $h_2h_1 \in H$ is a desired one.
\end{proof}

\begin{corollary}\label{cor:HI=GI}
  Assume that $G$ satisfies Conditions A'(1), A'(3), B'(1), and B'(2) with respect to $X$.
  Let $I_0 = [x_0, y_0]$ be an $X$-interval with $H|_{I_0} = G(I_0)$.
  Let $x_1$ and $y_1$ be elements of $G \cdot x_0$ and $G \cdot y_0$, respectively, with $x_1 < y_1$.
  We set $I = [x_1, y_1]$, which is an $X$-interval.
  Then, the equality $H(I) = G(I)$ holds.
\end{corollary}

\begin{proof}
  By Lemma \ref{lem:equal_X_interval}, we have $H(I_0) = G(I_0)$.
  By Lemma \ref{lem:double_trans}, there exists $h \in H$ such that $h(I_0) = I$.
  Hence, we obtain $H(I) = H(I_0)^{h} = G(I_0)^h = G(I)$.
\end{proof}

We are now ready to prove Theorem \ref{thm:conditions}.

\begin{proof}[Proof of Theorem $\ref{thm:conditions}$]
  Assume that $G$ satisfies Conditions A' and B' with respect to $X$.
  Let $H$ be a classful subgroup of $G$.
  We show that $G \subset H$.
  We set
  \[
    G_c = \{ g \in G \mid \Supp(g) \text{ is compact} \}
  \]
  and set $H_c = G_c \cap H$.
  By Condition B'(3), we take an $X$-interval $I_0 = [x_0, y_0]$ with $H|_{I_0} = G(I_0)$.
  By Condition A'(2), we take an increasing sequence $\{ I_n = [x_n, y_n] \}_{n \geq 1}$ of $X$-intervals satisfying $x_n \in G\cdot x_0$ and $y_n \in G\cdot y_0$ for every $n \geq 1$ and
  \[
    \bigcup_{n \geq 1} I_n = \RR.
  \]
  Then, we have $G_c = \bigcup_{n} G(I_n)$ and $H_c = \bigcup_{n}H(I_n)$.
  By Corollary \ref{cor:HI=GI}, we have $H(I_n) = G(I_n)$ for every $n \geq 1$, and hence $H_c = G_c$.

  Let $f$ be an element of $G$.
  Since $H$ is classful, there exists $g \in G$ such that $f^g \in H$.
  The fact that the commutator $f^{g} f^{-1} = [g,f]$ is contained in $G_c = H_c \subset H$, together with $f^g \in H$, implies that $f \in H$.
  This completes the proof.
\end{proof}

\section{Proof of Theorem \ref{mainthm} (1) and (4)}
In this section, we show Theorem \ref{mainthm} (1).
We set $G = \PP_1$ and $X = \RR$.
Note that if we replace $X$ with $\QQ$ and $G$ with the group of piecewise $\PSL(2,\ZZ)$ homeomorphisms of the real line, then the following arguments also give the proof of Theorem \ref{mainthm} (4).

It is clear that $G$ satisfies Conditions A'(1) and A'(3) with respect to $X$.
Since the map $f\colon \RR \to \RR$ defined by $f(t) = t+1$ is contained in $G$, Condition A'(2) is also satisfied.

\begin{proposition}\label{prop:B1}
  Let $H$ be a classful subgroup of $G$.
  Then the action of $H$ on $X$ is transitive.
  In particular, $G$ satisfies Condition B'(1) with respect to $X$.
\end{proposition}

\begin{proof}
  Let $H$ be a classful subgroup of $G$.
  For the element $f \in G$ defined by $f(t) = t+1$, there exist $g \in G$ such that $f^g \in H$ since $H$ is classful in $G$.
  Since $g$ is contained in $G$, the map $g$ is of the form $g(t) = t + a$ on both end linear zones.
  Hence the conjugate $f^g$ is of the form $f^g(t) = t+1$ on both end linear zones.
  Note that, as $f(t) > t$ for every $t \in \RR$, the map $f^g$ also satisfies $f^g(t) > t$ for every $t \in \RR$.
  We set $h_0 = f^g \in H$.

  Let $x$ be an element of $X$.
  Let $n$ be an integer satisfying $h_0(t) = t+1$ for every $t \in [n, \infty)$.
  Since the interval $[n, n+1]$ is a fundamental domain of the map $h_0$, there exist an integer $l$ and a real number $\alpha \in [0, 1)$ such that $h_0^l(x) = n+\alpha$.
  As with the existence of $h_0$, there exist $h_1 \in H$ and an integer $m$ greater than $n$ such that $h_1$ is of the form $h_1(t) = t - \alpha$ for every $t \in [m, \infty)$.
  Then the function $h = h_0^{n-m}h_1h_0^{l+m-n}$ is contained in $H$ and satisfies
  \[
    h(x) = h_0^{n-m}h_1h_0^{l+m-n}(x) = n.
  \]
  This implies that for every $x \in X$, there exists an integer $n$ such that $\ZZ_{\geq n} \subset H \cdot x$.
  Hence the lemma follows.
\end{proof}

\begin{proposition}
  Let $H$ be a classful subgroup of $G$ and $I = [v,\infty)$ an $X^*$-interval.
  Then the action of $H(I)$ on $(v, \infty)$ is transitive.
  In particular, Condition B'(2) is satisfied.
\end{proposition}

\begin{proof}
  We have seen that the group $G$ satisfies Conditions A'(1), A'(3), and B'(1) with respect to $X$.
  Hence, by Lemma \ref{lem:classful_inf_int}, the group $H(I)$ is classful in $G(I)$ for every $X^*$-interval of the form $[v,\infty)$.
  Therefore, the proposition follows from the same argument as in the proof of Proposition \ref{prop:B1}.
\end{proof}

Finally, we show that $G$ satisfies Condition B'(3) with respect to $X$.
Let $G_1$ be the subset of $G$ whose elements $g$ satisfy $g(t) > t$ for every $t \in \RR$ and are of the form $g(t) = t+1$ on both end linear zones.
For every real number $t \in \RR$, we set $J_t = [t, t+1]$.
Let $g$ be an element of $G_1$ and $J_a$ and $J_b$ intervals which lie in the union of the end linear zones of $g$, where $a, b \in \RR$.
Following \cite{MR4172581}, we call intervals $J_a$ and $J_b$ \emph{monitoring intervals for $g$} if there exists a positive integer $N$ such that $g^N(J_a) = J_b$.
For every real number $x \in \RR$, we define a homeomorphism $\phi_{x} \colon J_0 \to J_x$ by $\phi_{x}(t) = t+x$.
If $J_a$ and $J_b$ are monitoring intervals for $g$, the homeomorphism
\[
  f = \phi_b^{-1} g^N \phi_a \colon J_0 \to J_0
\]
is called the \emph{information of $g$ monitored by $J_a$ and $J_b$}.
For every $g \in G_1$, let $\I(g)$ be the set of information of $g$ monitored by some intervals $J_a$ and $J_b$.
Note that, if $J_a$ and $J_b$ are monitoring intervals of $g$ with information $f$ and if $J_a$ (resp. $J_b$) is contained in the end linear zone at $-\infty$ (resp. at $\infty$), then $J_{a-n}$ and $J_{b+m}$ are also monitoring intervals of $g$ with information $f$ for every positive integer $n$ and $m$.

\begin{lemma}\label{lem:any_inf_monitored}
  For every $f \in G(J_0)$, there exists an element $g$ of $G_1$ such that $f \in \I(g)$.
\end{lemma}

\begin{proof}
  Let $f$ be an element of $G(J_0)$.
  Then we define an element $g$ of $G_1$ by setting
  \[
    g(x) =
    \begin{cases}
      f(x)+1 & \ x \in J_0 \\
      x+1 & \text{ otherwise. }
    \end{cases}
  \]
  Then the information of $g$ monitored by $J_{-1}$ and $J_1$ is equal to $f$.
\end{proof}

\begin{lemma}\label{lem:inf_class_inv}
  For every $g \in G_1$ and $h \in G$, we have $\I(g^h) = \I(g)$.
\end{lemma}

\begin{proof}
  We show that $\I(g) \subset \I(g^h)$.
  Let $f$ be an element of $\I(g)$.
  We let monitoring intervals $J_a$ and $J_b$ of $g$ with information $f$ be far enough so that $g(J_a)$ (resp. $J_b$) lies in the end linear zone at $-\infty$ (resp. at $\infty$) of $h$.
  Then the interval $h(J_a)$ (resp. $h(J_b)$) lies in the end linear zone at $-\infty$ (resp. at $\infty$) of $g^h$.
  Moreover, these intervals are monitoring intervals of $g^h$ with information $f$.
  Indeed, if $f = \phi_b^{-1}g^N \phi_a$, then we have
  \[
    f = \phi_{b}^{-1} h^{-1} h g^N h^{-1} h \phi_{a} = \phi_{h(b)}^{-1} (g^h)^{N} \phi_{h(a)}.
  \]
  Here the last equality is deduced from the fact that $J_a$ and $J_b$ lie in the linear end zones of $h$.
  This implies that $\I(g) \subset \I(g^h)$.
  By applying this inclusion to $g^h \in G_1$ and $h^{-1} \in G$, we obtain
  \[
    \I(g^h) \subset \I((g^h)^{h^{-1}}) = \I(g).
  \]
  This completes the proof.
\end{proof}

Lemmas \ref{lem:any_inf_monitored} and \ref{lem:inf_class_inv} imply the following.

\begin{corollary}\label{cor:any_info_classful}
  Let $H$ be a classful subgroup of $G$.
  For every $f \in G(J_0)$, there exists an element $h$ of $H \cap G_1$ such that $f \in \I(h)$.
\end{corollary}

By Corollary \ref{cor:any_info_classful}, we take $h_0 \in H\cap G_1$ such that $\id \in \I(h_0)$.
Let $J_a$ and $J_b$ be monitoring intervals of $h_0$ with information $\id$.
By definition, there exists a positive integer $N_0$ such that $h_0^{N_0}(J_a) = J_b$ and
\begin{align*}
  \id = \phi_b^{-1} h_0^{N_0} \phi_a.
\end{align*}
Since $h_0 \in G_1$ and $J_a$ and $J_b$ lie in the end linear zones of $h_0$, we have
\begin{align}\label{info_id}
  \id = \phi_{h_0^n(b)}^{-1} h_0^{N_0 + 2n} \phi_{h_0^{-n}(a)} = \phi_{b+n}^{-1} h_0^{N_0 + 2n} \phi_{a-n}
\end{align}
for every positive integer $n$.

\begin{proposition}\label{prop:B'(3)_main1}
  For $J_a$ above, we have $H|_{J_a} = G(J_a)$.
  In particular, Condition B'(3) is satisfied.
\end{proposition}

\begin{proof}
  By Condition A'(3), we have $H|_{J_a} \subset G|_{J_a} = G(J_a)$.

  We show that $G(J_a) \subset H|_{J_a}$.
  Let $\hf$ be an element of $G(J_a)$, and set
  \[
    f = \phi_a^{-1} \hf \phi_a \in G(J_0).
  \]
  By Corollary \ref{cor:any_info_classful}, we take $h_1 \in H \cap G_1$ satisfying $f \in \I(h_1)$.
  Let $J_c$ and $J_d$ be monitoring intervals of $h_1$ with information $f$.
  By definition, there exists a positive integer $N_1$ such that $h_1^{N_1}(J_c) = J_d$ and
  \begin{align*}
    f = \phi_d^{-1} h_1^{N_1} \phi_c.
  \end{align*}
  As with (\ref{info_id}), we have
  \begin{align}\label{info_f}
    f = \phi_{h_1^n(d)}^{-1} h_1^{N_1} \phi_{h_1^{-n}(c)} = \phi_{d+n}^{-1} h_1^{N_1 + 2n} \phi_{c-n}
  \end{align}
  for every positive integer $n$.
  Then, by (\ref{info_id}) and (\ref{info_f}), we obtain
  \begin{align}\label{f=idf}
    f = \phi_{a-n}^{-1} h_0^{-N_0 - 2n} \phi_{b+n}  \phi_{d+n}^{-1} h_1^{N_1 + 2n} \phi_{c-n}.
  \end{align}
  Since $H$ is classful in $G$, there exists $h_2 \in H$ such that $h_2(x) = x - a + c$ (resp. $h_2(x) = x - b + d$) on the end linear zone at $-\infty$ (resp. at $\infty$) of $h_2$.
  We take $n$ large enough so that $J_{a-n}$ and $J_{b+n}$ lie in the end linear zones of $h_2$.
  Then we have $h_2 \phi_{a-n} = \phi_{c-n}$ and $h_2 \phi_{b+n} = \phi_{d+n}$.
  Hence, together with (\ref{f=idf}) and the equality $h_0^{n} \phi_{a-n} = \phi_a$, we obtain
  \begin{align*}
    f = \phi_a^{-1} h_0^{-N_0 - n} h_2^{-1} h_1^{N_1 + 2n} h_2 h_0^{-n} \phi_a.
  \end{align*}
  This implies that $\hf = (h_0^{-N_0 - n} h_2^{-1} h_1^{N_1 + 2n} h_2 h_0^{-n})|_{J_a}$.
  Hence $\hf$ is contained in $H|_{J_a}$.
\end{proof}

\begin{proof}[Proof of Theorem \textup{\ref{mainthm} (1)}]
  We have seen that the group $G$ satisfies Conditions A' and B' with respect to $X$.
  Hence, by Theorem \ref{thm:conditions}, the group $G$ is invariably generated.

  Let us consider the group extension
  \[
    1 \to G \to \PP \to \PP/G \to 1.
  \]
  Since $G$ contains the commutator subgroup of $\PP$, the quotient group $\PP/G$ is abelian.
  In particular, $\PP/G$ is invariably generated.
  Since invariable generation is closed under taking extension (\cite{MR417300}), the group $\PP$ is also invariably generated.
\end{proof}

\section{Proof of Theorem \ref{mainthm} (2) and (3)}

In this section, we show Theorem \ref{mainthm} (3). The proof of Theorem \ref{mainthm} (2) is similar.
Let $G$ be Monod's group $H(\ZZ)$ and set $X = P_{\ZZ}$.
By Theorem \ref{thm:conditions}, it suffices to show that the group $G$ satisfies Conditions A' and B' with respect to $X$.

The following proposition is clear.
\begin{proposition}\label{prop:A_main3}
  The group $G$ satisfies Condition A' with respect to $X$.
\end{proposition}

Let $G_1$ be the subset of $G$ whose elements $g$ satisfy $g(t) > t$ for every $t \in \RR$ and are of the form $g(t) = t+1$ on their end linear zones.
Let $x$ be an element of $X$ and set $X_x = (G \cdot x) \cap [0,1)$.

We define a map $\beta_{x} \colon G_1 \to X_x$ as follows.
For $g \in G_1$, we take an integer $n$ large enough so that the interval $[x-n, x-n+1]$ lies in the end linear zone of $g$ at $-\infty$.
Then it is a fundamental domain of $g$.
We take an integer $m \in \ZZ$ large enough so that the interval $g^m([x-n, x-n+1])$ lies in the end linear zone of $g$ at $\infty$.
Then there exist a unique $\alpha \in [0,1)$ and a unique integer $l \in \ZZ$ such that $g^m([x-n, x-n+1]) = [\alpha + l, \alpha + l+1]$.
Now we define the map $\beta_{x} \colon G_1 \to X_x$ by $\beta_x(g) = \alpha$.
Note that this $\beta_x$ is independent of the choice of $n$ and $m$.

\begin{lemma}\label{lem:beta_class_inv}
  For every $x \in X$, the map $\beta_x \colon G_1 \to X_x$ is class invariant, that is, the equality $\beta_x(g^f) = \beta_x(g)$ holds for every $f \in G$ and $g \in G_1$.
\end{lemma}

\begin{proof}
  Let $x$ be an element of $X$ and $f$ an element of $G$.
  Then there exist integers $i$ and $j$ such that the function $f$ is of the form $f(t) = t+i$ (resp. $f(t) = t+j$) on the end linear zone at $-\infty$ (resp. at $\infty$) of $f$.
  For an element $g \in G_1$, we set $\alpha = \beta_x(g)$ and take integers $l, m, n \in \ZZ$ such that
  \begin{enumerate}[$(1)$]
    \item $[x-n, x-n+1]$ and $[\alpha + l, \alpha + l + 1]$ lie in the end linear zone of $g$ at $-\infty$ and $\infty$, respectively,
    \item $[x-n+i, x-n+i+1]$ and $[\alpha + l+j, \alpha + l+j+1]$ lie in the end linear zone of $g^f$ at $-\infty$ and $\infty$, respectively, and
    \item the equality $g^m([x-n, x-n+1]) = [\alpha + l, \alpha + l+1]$ holds.
  \end{enumerate}
  Then, since the equality
  \[
    (g^f)^m([x-n+i, x-n+i+1]) = [\alpha + l+j, \alpha + l+j+1]
  \]
  holds, we have $\beta_x(g^f) = \alpha$.
\end{proof}

\begin{lemma}\label{lem:surj_main3}
  For every $x \in X$, the map $\beta_x \colon G_1 \to X_x$ is surjective.
\end{lemma}

\begin{proof}
  Let $y$ be an element of $X_x$.
  We take an integer $m \in \ZZ$ with $x+1 < y+m$.
  We define a map $g \in G_1$ as follows.
  For every $t \in \RR \setminus (x, y+m)$, we set $g(t) = t+1$.
  On the interval $[x, y+m]$, we set the map $g$ to be a piecewise $\PSL(2,\ZZ)$ homeomorphism with $g([x, x+1]) = [x+1, y+m]$ and $g([x+1, y+m]) = [y+m, y+m+1]$ whose breakpoints are finite and lie in $X$.
  We now verify that such a map from $[x, y+m]$ to $[x+1, y+m+1]$ exists.

  Since $y \in G \cdot x$, there exists an element $h$ of $\PSL(2,\ZZ)$ such that $h(x) = y$.
  First we consider how to define the map on $[x,x+1]$.
  If the graphs of $h \colon \RR \to \RR$ and $f_1 \colon \RR \to \RR$ defined by $f_1(t) = t+1$ intersect at $p \in [x, x+1]$ transversely, then we set $g = f_1$ on $[x, p]$ and $g = h$ on $[p, x+1]$.
  Note that the element $p$ is contained in $P_{\ZZ}$ since the graphs of $h$ and $f_1$ intersect transversely.

  Assume that the graphs do not intersect on $[x, x+1]$ or are tangent to each other at a point $z$ in $[x, x+1]$.
  Note that, in the latter case, $z$ is a rational point.
  Indeed, since $z$ is a tangent point of $h$ and $f_1$, it is a fixed point of a parabolic element of $\PSL(2,\ZZ)$.
  In particular, $z$ is not contained in $P_{\ZZ}$.
  Let $d$ be the unique integer contained in $(x, x+1)$.
  (Note that $x$ is not an integer since $x \in X = P_{\ZZ}$.)
  Take an integer $a$ with
  \[
    a > \frac{(y+m)(x+1-d)+1}{x+1-d}.
  \]
  We set $g_1 \colon \RR \cup \{ \infty \} \to \RR \cup \{ \infty \}$ as the M\"{o}bius transformation of
  \[
    \begin{pmatrix}
      a & -1-ad \\
      1 & -d
    \end{pmatrix}_{\textstyle .}
  \]
  Then we have $\displaystyle \lim_{t \to d + 0} g_1(t) = -\infty$ and $g_1(x+1) > y+m$.
  Hence, by the intermediate value theorem, the graphs of $f_1$ and $g_1$ (resp. of $g_1$ and $h$) intersect transversely at $p \in (d, x+1)$ (resp. at $q \in (d, x+1)$).
  By the transversality, the points $p$ and $q$ are contained in $X$.
  Since $z \notin X$, we have $p \neq z$ and $q \neq z$.
  Moreover, since $f_1(t) < h(t)$ for every $t \in [x, x+1]\setminus \{ z \}$, we have $p < q$.
  We set $g = f_1$ on $[x, p]$, $g = g_1$ on $[p, q]$, and $g = h$ on $[q, x+1]$.

  In the same vain, we can define a piecewise $\PSL(2,\ZZ)$ homeomorphism from $[x+1, y+m]$ to $[y+m, y+m+1]$ whose breakpoints are finite and lie in $P_{\ZZ}$.
\end{proof}

Lemmas \ref{lem:beta_class_inv} and \ref{lem:surj_main3} imply the following.

\begin{corollary}\label{cor:surj_classful_main3}
  Let $H$ be a classful subgroup of $G$.
  Then the map $\beta_x$ restricted to $H\cap G_1$ is also surjective onto $X_x$.
\end{corollary}

We now show Condition B'(1).
\begin{proposition}\label{prop:B(1)_main3}
  The group $G$ satisfies Condition B'(1) with respect to $X$.
\end{proposition}

\begin{proof}
  Let $H$ be a classful subgroup of $G$ and $x$ an element of $X$.
  Let $y$ be an element of $G \cdot x$.
  We take $h_1 \in H \cap G_1$ arbitrary and let $k \in \ZZ$ large enough so that $h_1^k(y)$ lies in the end linear zone of $h_1$ at $\infty$.
  We set $h_1^k(y) = \alpha + l$, where $\alpha \in X_x$ and $l \in \ZZ$.
  Then we have $h_1^{k+m}(y) = \alpha + l + m$ for every $m \in \ZZ_{\geq 0}$.
  By Corollary \ref{cor:surj_classful_main3}, we take $h_2 \in H\cap G_1$ with $\beta_x(h_2) = \alpha$.
  Then there exists $n_0 \in \ZZ$ such that
  \[
    \{ x-n \mid n \in \ZZ_{\geq n_0} \} \subset \{ h_2^{-n_2} h_1^{n_1}(y) \mid n_1, n_2 \in \ZZ \}.
  \]
  This implies that $H$ acts on $G \cdot x$ transitively.
  Since $H$ and $x$ are arbitrary, $G$ satisfies Condition B'(1) with respect to $X$.
\end{proof}

To prove Condition B'(2), we modify the definition of the map $\beta_x \colon G_1 \to X_x$ as follows.
First we recall that for every $x \in P_{\ZZ}$, the stabilizer $\PSL(2, \ZZ)_{x}$ is an infinite cyclic group.
To see this, let $x' \in P_{\ZZ}$ be the other fixed point of every element of $\PSL(2,\ZZ)_{x}$.
Then $\PSL(2, \ZZ)_{x}$ acts on the geodesic in the Poincar\'{e} disk connecting $x$ and $x'$.
Hence $\PSL(2,\ZZ)_{x}$ is a subgroup of $\RR$ which is discrete, that is, an infinite cyclic group.
Let $f_x$ be the generator of $\PSL(2, \ZZ)_{x}$ which satisfies $f_x(y) \geq y$ between $x$ and $x'$.

Let $I = [v, \infty)$ an $X^*$-interval.
Then, for every element $g$ of $G(I)$, its restriction to $[v, v + \varepsilon]$ for small positive $\varepsilon$ is equal to $f_v^n$ for some $n \in \ZZ$.
Let $G(I)_1$ be the subset of $G(I)$ whose elements $g$ satisfy $g(t) > t$ for every $t \in (v,\infty)$ and are of the form $g(t) = t + 1$ on the end linear zone at $\infty$ and of the form $g = f_v$ on $[v, v + \varepsilon]$ for small positive $\varepsilon$.
We fix an element $g_0 \in G(I)_1$.
Let $x$ be an element of $I \cap X$ and set $K_n = [g_0^n(x), g_0^{n+1}(x)]$ for every $n \in \ZZ$.
Note that each $K_n$ is a fundamental domain of $g_0|_{I} \colon I \to I$ satisfying $g_0(K_n) = K_{n+1}$.

Let $g$ be an element of $G$ satisfying $\Supp(g) \subset I$ and $g|_{I} \in G(I)_1$.
Let $v_1 \in X$ be the second smallest breakpoint of $g$.
(The smallest breakpoint of $g$ is $v$.)
Let $n$ be an integer satisfying $K_{n} \subset [v, v_1]$.
Then $K_{n}$ is a fundamental domain of $g|_{I}$.
Let $m$ be an integer such that $g^m(K_{n})$ lies in the end linear zone of $g$.
Then, since $g \in G_1$, there uniquely exist $\alpha \in X_x$ and $l \in \ZZ$ such that $g^m(K_{n}) = [\alpha + l, \alpha + l +1]$.
We define $\gamma_x \colon G(I)_1 \to X_x$ by $\gamma_x(g) = \alpha$.
Note that $\gamma_x$ is independent of the choice of $n$ and $m$.

The proof of the following lemma is the same as that of Lemma \ref{lem:beta_class_inv}.
\begin{lemma}\label{lem:gamma_class_inv}
  For every $x \in I \cap X$, the map $\gamma_x$ is class invariant, that is, the equality $\gamma_x(g^f) = \gamma_x(g)$ holds for every $f \in G(I)$ and $g \in G(I)_1$.
\end{lemma}

\begin{lemma}\label{lem:gamma_surj}
  For every $x \in I \cap X$, the map $\gamma_x \colon G(I)_1 \to X_x$ is surjective.
\end{lemma}

\begin{proof}
  Let $y$ be an element of $X_x$.
  We take an integer $n \in \ZZ$ large enough so that $g_0^n(x)$ is contained in the end linear zone of $g_0$ at $\infty$.
  We set $x_0 = g_0^n(x)$.
  Let $m$ be an integer with $x_0 + 1 < y + m$.
  Then the same argument as in the proof of Lemma \ref{lem:surj_main3} shows that $\gamma_x$ is surjective.
\end{proof}

\begin{corollary}\label{cor:gamma_HI_surj}
  Let $H$ be a classful subgroup of $G$.
  Then, the map $\gamma_x$ restricted to $H(I) \cap G(I)_1$ is surjective onto $X_x$.
\end{corollary}

\begin{proof}
  By Lemma \ref{lem:classful_inf_int}, the group $H(I)$ is classful in $G(I)$.
  Hence, Lemmas \ref{lem:gamma_class_inv} and \ref{lem:gamma_surj} imply the corollary.
\end{proof}

\begin{proposition}\label{prop:B2_main3}
  The group $G$ satisfies Condition B'(2) with respect to $X$.
\end{proposition}

\begin{proof}
  Let $H$ be a classful subgroup of $G$ and $I = [v, \infty)$ an $X^*$-interval.
  Let $x$ be an element of $I \cap X$.
  By Lemma \ref{lem:classful_inf_int}, the group $H(I)$ is classful in $G(I)$.

  Let $y$ be an element of $(G \cdot x) \cap (v, \infty)$.
  We take $h_1 \in H(I) \cap G(I)_1$ arbitrary and an integer $k$ large enough so that $h_1^k(y)$ is contained in the end linear zone of $h_1$ at $\infty$.
  Then there uniquely exist $\alpha \in X_x$ and $l \in \ZZ$ such that for every $m \in \ZZ$, the equality $h_1^{k+m}(y) = \alpha + l + m$ holds.
  By Corollary \ref{cor:gamma_HI_surj}, we take $h_2 \in H(I) \cap G(I)_1$ with $\gamma_x(h_2) = \alpha$.
  Then there exists $n_0 \in \ZZ$ such that
  \[
    \{ g_0^{-n}(x) \mid n \in \ZZ_{\geq n_0} \} \subset \{ h_2^{-n_2} h_1^{n_1}(y) \mid n_1, n_2 \in \ZZ \}.
  \]
  This implies that $H(I)$ acts on $(G \cdot x) \cap (v, \infty)$ transitively.
\end{proof}

Finally, we show that $G$ satisfies Condition B'(3) with respect to $X$.
The proof is a slight modification of the proof of Proposition \ref{prop:B'(3)_main1}.

Let $x$ be an element of $X$.
For every integer $n \in \ZZ$, we set $J_{x, n} = [x+n, x+n+1]$.
Let $g$ be an element of $G_1$ and $J_{x,a}$ and $J_{x,b}$ intervals which lie in the union of the end linear zones of $g$, where $a, b \in \ZZ$.
Intervals $J_{x,a}$ and $J_{x,b}$ are called \emph{monitoring intervals} for $g$ if there exists a positive integer $N$ such that $g^N(J_{x,a}) = J_{x,b}$.
For every integer $n \in \ZZ$, we define a homeomorphism $\phi_{x,n} \colon J_{x,0} \to J_{x,n}$ by $\phi_{n}(t) = t+n$.
If $J_{x,a}$ and $J_{x,b}$ are monitoring intervals for $g$, the homeomorphism
\[
  f = \phi_{x,b}^{-1} g^N \phi_{x,a} \colon J_{x,0} \to J_{x,0}
\]
is called the \emph{information of $g$ monitored by $J_{x,a}$ and $J_{x,b}$}.
For every $g \in G_1$, let $\I_x(g)$ be the set of information of $g$ monitored by some intervals $J_{x,a}$ and $J_{x,b}$.
Note that, if $J_{x,a}$ and $J_{x,b}$ are monitoring intervals of $g$ with information $f$, then so are $J_{x,a-n}$ and $J_{x,b+m}$ for every positive integer $n$ and $m$.

\begin{lemma}\label{lem:any_inf_monitored_main3}
  For every $f \in G(J_{x,0})$, there exists an element $g$ of $G_1$ such that $f \in \I_x(g)$.
\end{lemma}

\begin{proof}
  Let $f$ be an element of $G(J_{x,0})$.
  Then we define an element $g$ of $G_1$ by setting
  \[
    g(t) =
    \begin{cases}
      f(t)+1 & \ t \in J_{x,0} \\
      t+1 & \text{ otherwise. }
    \end{cases}
  \]
  Then the information of $g$ monitored by $J_{x,-1}$ and $J_{x,1}$ is equal to $f$.
\end{proof}

\begin{lemma}\label{lem:inf_class_inv_main3}
  For every $g \in G_1$ and $h \in G$, we have $\I_x(g^h) = \I_x(g)$.
\end{lemma}

\begin{proof}
  We show that $\I_x(g) \subset \I_x(g^h)$.
  Let $f$ be an element of $\I_x(g)$.
  We let monitoring intervals $J_{x,a}$ and $J_{x,b}$ of $g$ with information $f$ far enough so that $g(J_{x,a})$ (resp. $J_{x,b}$) lies in the end linear zone at $-\infty$ (resp. at $\infty$) of $h$.
  Then the interval $h(J_{x,a})$ (resp. $h(J_{x,b})$) lies in the end linear zone at $-\infty$ (resp. at $\infty$) of $g^h$.
  Moreover, these intervals are monitoring intervals of $g^h$ with information $f$.
  Indeed, let $l$ and $m$ be integers with $h(t) = t + l$ and $h(t) = t + m$ on the end linear zone of $h$ at $-\infty$ and at $\infty$, respectively.
  If $f = \phi_{x,b}^{-1}g^N \phi_{x,a}$, then we have
  \[
    f = \phi_{x,b}^{-1} h^{-1} h g^N h^{-1} h \phi_{x,a} = \phi_{x, b+m}^{-1} (g^h)^{N} \phi_{x, a+l}.
  \]
  Here the last equality is deduced from the fact that $J_{x,a}$ and $J_{x,b}$ lie in the linear end zones of $h$.
  This implies that $\I_x(g) \subset \I_x(g^h)$.
  By applying this inclusion to $g^h \in G_1$ and $h^{-1} \in G$, we obtain
  \[
    \I_x(g^h) \subset \I_x((g^h)^{h^{-1}}) = \I_x(g).
  \]
  This completes the proof.
\end{proof}

\begin{corollary}\label{cor:any_info_classful_main3}
  Let $H$ be a classful subgroup of $G$.
  For every $f \in G(J_{x,0})$, there exists an element $h$ of $H \cap G_1$ such that $f \in \I_x(h)$.
\end{corollary}

\begin{proof}
  This is a consequence of Lemmas \ref{lem:any_inf_monitored_main3} and \ref{lem:inf_class_inv_main3}.
\end{proof}

By Corollary \ref{cor:any_info_classful_main3}, we take $h_0 \in H\cap G_1$ such that $\id \in \I_x(h_0)$.
Let $J_{x,a}$ and $J_{x,b}$ be monitoring intervals of $h_0$ with information $\id$.
By definition, there exists a positive integer $N_0$ such that $h_0^{N_0}(J_{x,a}) = J_{x,b}$ and
\begin{align*}
  \id = \phi_{x,b}^{-1} h_0^{N_0} \phi_{x,a}.
\end{align*}
Since $h_0 \in G_1$ and $J_{x,a}$ and $J_{x,b}$ lie in the end linear zones of $h_0$, we have
\begin{align}\label{info_id_main3}
  \id = \phi_{x,h_0^n(b)}^{-1} h_0^{N_0 + 2n} \phi_{x,h_0^{-n}(a)} = \phi_{x,b+n}^{-1} h_0^{N_0 + 2n} \phi_{x,a-n}
\end{align}
for every positive integer $n$.

\begin{proposition}\label{prop:B3_main3}
  For $J_{x,a}$ above, we have $H|_{J_{x,a}} = G(J_{x,a})$.
  In particular, Condition B'(3) is satisfied.
\end{proposition}

\begin{proof}
  By Condition A'(3), we have $H|_{J_{x,a}} \subset G|_{J_{x,a}} = G(J_{x,a})$.

  We show that $G(J_{x,a}) \subset H|_{J_{x,a}}$.
  Let $\hf$ be an element of $G(J_{x,a})$, and set
  \[
    f = \phi_{x,a}^{-1} \hf \phi_{x,a} \in G(J_{x,0}).
  \]
  By Corollary \ref{cor:any_info_classful_main3}, we take $h_1 \in H \cap G_1$ satisfying $f \in \I_x(h_1)$.
  Let $J_{x,c}$ and $J_{x,d}$ be monitoring intervals of $h_1$ with information $f$.
  By definition, there exists a positive integer $N_1$ such that $h_1^{N_1}(J_{x,c}) = J_{x,d}$ and
  \begin{align*}
    f = \phi_{x,d}^{-1} h_1^{N_1} \phi_{x,c}.
  \end{align*}
  As with (\ref{info_id_main3}), we have
  \begin{align}\label{info_f_main3}
    f = \phi_{x,h_1^n(d)}^{-1} h_1^{N_1} \phi_{x,h_1^{-n}(c)} = \phi_{x,d+n}^{-1} h_1^{N_1 + 2n} \phi_{x,c-n}
  \end{align}
  for every positive integer $n$.
  Then, by (\ref{info_id_main3}) and (\ref{info_f_main3}), we obtain
  \begin{align}\label{f=idf_main3}
    f = \phi_{x,a-n}^{-1} h_0^{-N_0 - 2n} \phi_{x,b+n}  \phi_{x,d+n}^{-1} h_1^{N_1 + 2n} \phi_{x,c-n}.
  \end{align}
  Since $H$ is classful in $G$, there exists $h_2 \in H$ such that $h_2(t) = t - a + c$ (resp. $h_2(t) = t - b + d$) on the end linear zone at $-\infty$ (resp. at $\infty$) of $h_2$.
  We take $n$ large enough so that $J_{x,a-n}$ and $J_{x,b+n}$ lie in the end linear zones of $h_2$.
  Then we have $h_2 \phi_{x,a-n} = \phi_{x,c-n}$ and $h_2 \phi_{x,b+n} = \phi_{x,d+n}$.
  Hence, together with (\ref{f=idf_main3}) and the equality $h_0^{n} \phi_{x,a-n} = \phi_{x,a}$, we obtain
  \begin{align*}
    f = \phi_{x,a}^{-1} h_0^{-N_0 - n} h_2^{-1} h_1^{N_1 + 2n} h_2 h_0^{-n} \phi_{x,a}.
  \end{align*}
  This implies that $\hf = (h_0^{-N_0 - n} h_2^{-1} h_1^{N_1 + 2n} h_2 h_0^{-n})|_{J_{x,a}}$.
  Hence $\hf$ is contained in $H|_{J_{x,a}}$.
\end{proof}

\begin{proof}[Proof of Theorem \textup{\ref{mainthm} (3)}]
  Theorem \ref{thm:conditions}, together with Propositions \ref{prop:A_main3}, \ref{prop:B(1)_main3}, \ref{prop:B2_main3} and \ref{prop:B3_main3}, concludes that $G$ is invariably generated.
\end{proof}

\section*{Acknowledgments}
The author would like to thank Yoshifumi Matsuda for discussions and comments.
The author is supported by JSPS KAKENHI Grant Number JP21J11199.

\appendix

\section{Invariable generation of groups of PL homeomorphisms of $[0,1]$}\label{sec:appendix}
In this appendix, we consider invariable generation of groups of PL homeomorphisms of the closed interval $[0,1]$. 
The goal of this appendix is to show Theorem \ref{thm:appendix_main}.

\subsection{Conditions for invariable generation}
In \cite{MR4172581}, Matsuda and Matsumoto provided conditions for groups of PL homeomorphisms of $[0,1]$ to be invariably generated.
Moreover, they showed that the Richard Thompson group $F$ (and certain subgroups of the group $\PL$ of PL homeomorphisms of $[0,1]$) satisfies the conditions.
The conditions are as follows.
Let $G$ be a subgroup of the group $\PL$ of orientation preserving PL homeomorphisms of the interval $[0,1]$.
For a subset $X$ of the open interval $(0,1)$, we set $X^* = X \cup \{ 0, 1 \}$.
A closed interval $I$ in $[0,1]$ is called an $X$-interval (resp., $X^*$-interval) if the endpoints of $I$ are contained in $X$ (resp., $X^*$).
For a subgroup $H$ of $G$ and an $X^*$-interval $I$, we set
\begin{align*}
  H(I) &:= \{ h|_I \colon h \in H, \Supp(h) \subset I \}, \\
  H|_I &:= \{ h|_I \colon h \in H, h(I) = I \}.
\end{align*}

\begin{definition}[\cite{MR4172581}]\label{def:MMCondA}
  Let $G$ be a subgroup of $\PL$ and $X$ a subset of $(0,1)$.
  We say that \emph{$G$ satisfies Condition A with respect to $X$} if the following hold:
  \begin{enumerate}
    \item[A$(1)$] $X$ is a dense $G$-invariant subset of $(0,1)$ which contains all the breakpoints of every element of $G$.
    \item[A$(2)$] $G$ acts on $X$ transitively.
    \item[A$(3)$] For every $X^*$-interval $I$, the equality $G|_I = G(I)$ holds.
    \item[A$(4)$] For every $X^*$-interval $I$, there exists a PL homeomorphism $\psi_I \colon [0,1] \to I$ such that $\psi_I(X^*) = X^* \cap I$ and $G^{\psi_I} = G(I)$, where $G^{\psi_I} = \{ g^{\psi_I} \colon I \to I \mid g \in G \}$.
  \end{enumerate}
\end{definition}

\begin{definition}[\cite{MR4172581}]\label{def:MMCondB}
  Let $G$ be a subgroup of $\PL$ and $X$ a subset of $(0,1)$.
  We say that \emph{$G$ satisfies Condition B with respect to $X$} if the following hold:
  \begin{enumerate}
    \item[B$(1)$] Every classful subgroup $H$ of $G$ acts on $X$ transitively.
    \item[B$(2)$] For every classful subgroup $H$, there exists an $X$-interval $I_0$ such that $H|_{I_0} = G(I_0)$.
  \end{enumerate}
\end{definition}

\begin{theorem}[\cite{MR4172581}]\label{thm:MMthm_cond_inv_gen}
  Let $G$ be a subgroup of $\PL$.
  If there exists $X \subset (0,1)$ such that $G$ satisfies Conditions A and B with respect to $X$, then $G$ is invariably generated.
\end{theorem}

Matsuda and Matsumoto showed that the group $F$ satisfies Conditions A and B with respect to the set of dyadic rationals $\ZZ[1/2] \cap (0,1)$, and concluded that $F$ is invariably generated.

Theorem \ref{thm:MMthm_cond_inv_gen} is not applicable to the Higman--Thompson group $F_n$ for $n \geq 3$ and $X = \ZZ[1/n] \cap (0,1)$ since $F_n$ does not satisfy Condition A(2) and B(2) with respect to $\ZZ[1/n] \cap (0,1)$.
Instead, we show in Subsection \ref{sec:appendix_Higman} that the Higman--Thompson group $F_n$ satisfies the following weaker forms (Conditions A' and B') of Conditions A and B with respect to $\ZZ[1/n] \cap (0,1)$, which also induces the invariable generation of $F_n$.

\begin{definition}
  Let $G$ be a subgroup of $\PL$ and $X$ a subset of $(0,1)$.
  We say that \emph{$G$ satisfies Condition A' with respect to $X$} if the following hold:
  \begin{enumerate}
    \item[A'$(1)$] $X$ is a $G$-invariant subset of $(0,1)$ which contains all the breakpoints of every element of $G$.
    \item[A'$(2)$] For every $x \in X$, the point $0$ and $1$ are accumulation points of the orbit $G\cdot x \subset [0,1]$.
    \item[A'$(3)$] For every $X^*$-interval $I$, the equality $G|_I = G(I)$ holds.
  \end{enumerate}
\end{definition}

\begin{definition}
  Let $G$ be a subgroup of $\PL$ and $X$ a subset of $[0,1]$.
  We say that \emph{$G$ satisfies Condition B' with respect to $X$} if the following hold:
  \begin{enumerate}
    \item[B'$(1)$] For every classful subgroup $H$ of $G$ and every $x \in X$, the action of $H$ on the orbit $G \cdot x$ is transitive.
    \item[B'$(2)$] For every classful subgroup $H$ of $G$, every $X^*$-interval $I$ of the form $[v, 1]$, and every $x \in I\cap X$, the action of $H(I)$ on $(G \cdot x) \cap (v,1]$ is transitive.
    \item[B'$(3)$] For every classful subgroup $H$ of $G$, there exists an $X$-interval $I_0$ such that $H|_{I_0} = G(I_0)$.
  \end{enumerate}
\end{definition}

Conditions A' and B' provide a criterion for the invariable generation of $G$ as follows.

\begin{theorem}\label{thm:conditions_PL}
  Let $G$ be a subgroup of $\PL$.
  If there exists $X \subset (0,1)$ such that $G$ satisfies Conditions A' and B' with respect to $X$, then $G$ is invariably generated.
\end{theorem}

We omit the proof because it was essentially given in the proof of Theorem \ref{thm:MMthm_cond_inv_gen} in \cite{MR4172581}, and it is also almost the same as that of Theorem \ref{thm:conditions}.

\subsection{Proof of Theorem \ref{thm:appendix_main} (1)}\label{sec:appendix_Higman}

By using Theorem \ref{thm:conditions_PL}, we show Theorem \ref{thm:appendix_main} (1).

From now on, let $n$ be an integer greater than one.
The Higman--Thompson group $F_n$ is the group of PL homeomorphisms of the closed interval $[0,1]$ whose slopes are powers of $n$ and breakpoints are contained in $n$-adic rationals $\ZZ[1/n]$.
(See \cite{MR885095} for properties of $F_n$.)
We set $X = \ZZ[1/n] \cap (0,1)$.
The following is clear.
\begin{proposition}\label{prop:appendix_A}
  The group $F_n$ satisfies Condition A' with respect to $X$.
\end{proposition}

Let $\alpha \colon F_n \to \ZZ^2$ be the map defined by
\[
  \alpha(f) = (\log_n f'(0), \log_n f'(1)).
\]
Let $F_{n,1,-1}$ be the subset of $F_n$ whose elements $f$ satisfy $f(t) > t$ for every $t \in (0,1)$ and $\alpha(f) = (1,-1)$.

For every $j \in \ZZ_{\geq 0}$, let $\Phi_{j} \colon [1-n^{-j}, 1-n^{-j-1}] \to [0,1]$ be the orientation preserving affine homeomorphism.
For $x \in X$, we set $X_x = \Phi_j\big((F_n \cdot x)\cap [1-n^{-j}, 1-n^{-j-1}]\big)$.
Note that $X_x$ is independent of the choice of $j \in \ZZ_{\geq 0}$.

We fix an element $g_0 \in F_{n,1,-1}$ and take an element $x$ of $X$.
Note that $g_0^{-n}(x)$ tends to $0$ as $n$ tends to $\infty$.
Let $f$ be an element of $F_{n,1,-1}$.
As with piecewise projective homeomorphisms, the \emph{end linear zone at $0$ \textup{(}resp. at $1$\textup{)}} of an element $f \in \PL$ is the maximal interval $[0, t]$ (resp. $[1-t, 1]$) on which $f$ is affine.
Let $i$ be an integer satisfying that the interval $[g_0^{-i-1}(x), g_0^{-i}(x)]$ is contained in the end linear zone of $f$ at $0$.
Then the interval $[g_0^{-i-1}(x), g_0^{-i}(x)]$ is a fundamental domain of $f$.
Let $m$ be an integer such that $f^m(g_0^{-i}(x))$ lies in the end linear zone of $f$ at $1$.
Then there exists a unique $j \in \ZZ_{\geq 0}$ such that $f^m(g_0^{-i}(x)) \in (1-n^{-j}, 1-n^{-j-1}]$.
We define $\beta_x \colon F_{n,1,-1} \to X_x$ by
\[
  \beta_x (f) = \Phi_{j}\big(f^m(g_0^{-i}(x))\big).
\]
This $\beta_x$ does not depend on the choice of $i, m$ and $j$ but depends on the choice of $g_0$.

\begin{lemma}\label{lem:appendix_beta_class_inv}
  The map $\beta_x \colon F_{n,1,-1} \to X_x$ is class invariant, that is, it satisfies $\beta_x(f^g) = \beta_x{f}$ for every $f \in F_{n,1,-1}$ and $g \in F_n$.
\end{lemma}

We omit the proof because it is similar to that of Lemma \ref{lem:beta_class_inv}.

\begin{lemma}\label{lem:appendix_beta_surj}
  The map $\beta_x \colon F_{n,1,-1} \to X_x$ is surjective.
\end{lemma}

\begin{proof}
  Let $y$ be an element of $X_x$.
  We take an integer $i$ satisfying that $g_0^{-i+1}(x)$ is contained in the end linear zone of $g_0$ at $0$ and that $\Phi_i^{-1}(y)$ is greater than $g_0^{-i+1}(x)$.
  We define a PL homeomorphism $f \colon [0,1] \to [0,1]$ as follows.
  On $[0,g_0^{-i}(x)] \cup [\Phi_{i}^{-1}(y),1]$, the map $f$ is defined by
  \[
    f(t) =
    \begin{cases}
      nt & t \in [0, g_0^{-i}(x)] \\
      1 + n^{-1}(t-1) & t \in [\Phi_{i}^{-1}(y),1].
    \end{cases}
  \]
  We define the map $f$ on $[g_0^{-i}(x),\Phi_{i}^{-1}(y)]$ to satisfy $f(g_0^{-i+1}(x)) = \Phi_{i}^{-1}(y)$ whose slopes are in $n^{\ZZ}$ and breakpoints are contained in $\ZZ[1/n]$.
  Considering that the points $x$ and $\Phi_{i}^{-1}(y)$ are contained in a single orbit of $G$, we can construct such a homeomorphism $[g_0^{-i}(x),\Phi_{i}^{-1}(y)] \to [g_0^{-i+1}(x), \Phi_{i+1}^{-1}(y)]$ by the argument analogous to \cite[Lemma 4.2]{cannon_floyd_parry}.
  The resulting homeomorphism $f$ is an element of $G_{1,-1}$ and satisfies $\beta_x(f) = y$.
\end{proof}

Lemmas \ref{lem:appendix_beta_class_inv} and \ref{lem:appendix_beta_surj} imply the following.
\begin{corollary}
  Let $H$ be a classful subgroup of $F_n$ and $x$ an element of $X$.
  Then the map $\beta_x$ restricted to $H\cap F_{n,1,-1}$ is surjective onto $X_x$.
\end{corollary}

The following is proved in a similar manner as in the proof of Proposition \ref{prop:B(1)_main3}.
\begin{proposition}\label{prop:appendix_B1}
  The group $F_n$ satisfies Condition B'(1) with respect to $X$.
\end{proposition}

In the same way, we obtain the following.
\begin{proposition}\label{prop:appendix_B2}
  The group $F_n$ satisfies Condition B'(2) with respect to $X$.
\end{proposition}

The following is proved in the same manner as in the proof of Condition B(2) for the Thompson group $F$ given in \cite{MR4172581}.
\begin{proposition}\label{prop:appendix_B3}
  The group $F_n$ satisfies Condition B'(3) with respect to $X$.
\end{proposition}

\begin{proof}[Proof of Theorem \textup{\ref{thm:appendix_main} (1)}]
  Theorem \ref{thm:conditions_PL},
  together with Propositions \ref{prop:appendix_A}, \ref{prop:appendix_B1}, \ref{prop:appendix_B2} and \ref{prop:appendix_B3},
  concludes that the Higman--Thompson group $F_n$ is invariably generated.
\end{proof}

\subsection{Proof of Theorem \ref{thm:appendix_main} (2)}
The \emph{golden ratio Thompson group} $F_{\tau}$ is the group of PL homeomorphisms of the interval $[0,1]$ whose breakpoints are in $\ZZ[\tau]$ and slopes are powers of $\tau$, where $\tau = (\sqrt{5}-1)/2$.
This group was introduced in \cite{MR1772420}.

In this section, we set $X = \ZZ[\tau] \cap (0,1)$.
We show that the group $F_{\tau}$ satisfies the Conditions A and B with respect to $X$ and use Theorem \ref{thm:MMthm_cond_inv_gen} to conclude its invariable generation.


\begin{proposition}\label{prop:appendix_A_gold}
  The group $F_{\tau}$ satisfies Condition A with respect to $X$.
\end{proposition}

\begin{proof}
  Since $\tau^2 = 1 - \tau$, every element of $X$ can be written as $a + b\tau$ for some $a, b \in \ZZ$.
  Hence $X$ is a dense $F_{\tau}$-invariant subset of $(0,1)$.
  By the definition of $F_{\tau}$, all the breakpoints of every $f \in F_{\tau}$ are contained in $X$.
  Hence Condition A(1) holds.
  Condition A(2) is obtained from \cite[Corollary 1]{MR1772420} (see also \cite[Lemma 3.5]{2110.06286}).
  Note that, as with the Thompson group $F$ case, Condition A(2) can be shown by using the tree diagram (see \cite[Lemma 4.2]{cannon_floyd_parry} and \cite{MR4278765}).
  Condition A(3) is clear.
  For Condition A(4), every map $\psi_I$ whose breakpoints are contained in $\ZZ[\tau]$ and slopes are powers of $\tau$ is a desired one.
  Such a map exists by the same argument as in \cite[Lemma 4.2]{cannon_floyd_parry} or \cite[Corollary 1]{MR1772420}.
\end{proof}

We set $U = \{ (a,b) \in \ZZ^2 \colon 0 \leq a < b \}$.
\begin{lemma}\label{lem:unique_rep}
  For every $p \in X$, there exists a unique non-negative integer $j \in \ZZ_{\geq 0}$ and a unique element $(a,b) \in U$ such that
  \[
    p = (a + b\tau) \tau^j.
  \]
\end{lemma}

The proof of this lemma is almost the same as in the proof of \cite[Lemma 1]{MR1772420}.

\begin{proof}[Proof of Lemma \textup{\ref{lem:unique_rep}}]
  For $p \in X \subset \ZZ[\tau]$, we write $p = a_0 + b_0 \tau$ for (uniquely determined) $a_0, b_0 \in \ZZ$.
  By the equality $1 = \tau + \tau^2$, we can write $p$ as
  \[
    p = (a_j + b_j \tau) \tau^j
  \]
  for every $j \in \ZZ_{\geq 0}$, where $a_j, b_j \in \ZZ$.
  We note that for $j > 0$, the integers $a_j$ and $b_j$ are determined by
  \[
  \left(
   \begin{array}{c}
    a_j \\
    b_j
   \end{array}
  \right)
  =
   \left(
    \begin{array}{cc}
     1 & 1  \\
     1 & 0
    \end{array}
   \right)^j
   \cdot
   \left(
    \begin{array}{c}
     a_0 \\
     b_0
    \end{array}
   \right)
   =
   \left(
    \begin{array}{cc}
     r_{j+1} & r_j  \\
     r_j & r_{j-1}
    \end{array}
   \right)
   \cdot
   \left(
    \begin{array}{c}
     a_0 \\
     b_0
    \end{array}
   \right),
  \]
  where $\{ r_n \}_{n \geq 0}$ is the Fibonacci sequence with $r_{0} = 0$ and $r_1 = 1$.
  For a large $j$, the integers $a_j$ and $b_j$ become positive.
  Indeed, since $a_0 + b_0\tau > 0$ and the ratio $r_{j+1}/r_j$ tends to $\tau^{-1}$ as $j$ tends to infinity, we have
  \[
    a_j = r_{j+1}a_0 + r_jb_0 \sim r_{j+1} (a_0 + b_0 \tau) > 0
  \]
  and
  \[
    b_j = r_j a_0 + r_{j-1} b_0 \sim r_j (a_0 + b_0 \tau) > 0.
  \]
  If $a_j,b_j \geq 0$, then $a_k, b_k \geq 0$ for every $k \geq j$.
  Note that exactly one of $a_0$ and $b_0$ is positive since $0 < a_0 + b_0\tau < 1$.
  Hence there exists a unique $j \in \ZZ_{\geq 0}$ such that there are no $i < j$ with $a_i, b_i \geq 0$.
  For this $j \in \ZZ_{\geq 0}$, the tuple $(a_j, b_j)$ is contained in $U$.
  Indeed, $a_j - b_j = b_{j-1}$ must be negative if $j>0$.
  If $j = 0$, then $(a_0, b_0) = (0,1) \in U$ since $0 < a_0 + b_0\tau < 1$ and $\tau > 1/2$.
  Moreover, for every $k\neq j$, the tuple $(a_k, b_k)$ is not contained in $U$.
  These uniquely determined $j, a_j$ and $b_j$ are the desired $j, a$ and $b$ in the statement.
\end{proof}

Let $\alpha \colon F_{\tau} \to \ZZ^2$ be a map defined by
\[
  \alpha(f) = (\log_{\tau^{-1}}f'(0), \log_{\tau^{-1}}f'(1)).
\]
This is a homomorphism and a class function, i.e., $\alpha(f^g) = \alpha(f)$ for every $f,g \in F_{\tau}$.
We set
\[
  F_{\tau, 1, -1} = \{ f \in F_{\tau} \colon \alpha(f) = (1,-1) \text{ and } f(x) > x \text{ for every } x \in (0,1) \}.
\]

For $g \in F_{\tau, 1, -1}$, we take a point of the form $\tau^l$ contained in the end linear zone of $g$ at $0$.
Let us consider the $g$-orbit of $\tau^l$.
Since $g$ is strictly monotonously increasing, $g^n(\tau^l)$ tends to $1$ as $n$ tends to $\infty$.
We take the smallest integer $n$ such that $g^n(\tau^l)$ is contained in the end linear zone of $g$ at $1$.
Then by Lemma \ref{lem:unique_rep}, there exists a unique non-negative integer $j_g \in \ZZ_{\geq 0}$ and a unique element $(a_g,b_g) \in U$ such that
\[
  g^n(\tau^{l}) = 1 - (a_g + b_g \tau) \tau^{j_g}.
\]
Note that $j_g, a_g$ and $b_g$ are independent of the choice of $l$ since such $\tau^l$'s are contained in the same $g$-orbit.
Hence we can obtain a map $\beta \colon F_{\tau, 1, -1} \to U$ by setting
\[
  \beta(g) = (a_g,b_g).
\]

By using this $\beta$, we can show the following in the same way as in the case of the Thompson group $F$ given in \cite{MR4172581}.
\begin{proposition}\label{prop:appendix_B_gold}
  The group $F_{\tau}$ satisfies Condition B with respect to $X$.
\end{proposition}

\begin{proof}[Proof of Theorem \textup{\ref{thm:appendix_main} (2)}]
  Theorem \ref{thm:MMthm_cond_inv_gen}, together with Propositions \ref{prop:appendix_A_gold} and \ref{prop:appendix_B_gold}, concludes that $F_{\tau}$ is invariably generated.
\end{proof}

\bibliographystyle{amsalpha}
\bibliography{references.bib}
\end{document}